\documentclass{amsart}

\newtheorem{Theorem}{Theorem}

\newtheorem{Lemma}{Lemma}
\newtheorem{Corollary}{Corollary}
\newtheorem{Claim}{Claim}

\theoremstyle{remark}

\begin{document}

\title[Degeneracy slopes, $\partial$-slopes and exceptional surgery slopes]{Degeneracy slopes, boundary slopes and exceptional surgery slopes}

\author{Kazuhiro Ichihara}

\address{Department of Mathematics, College of Humanities and Sciences, Nihon University, 3--25--40 Sakurajosui, Setagaya, Tokyo 156--8550, Japan}

\email{ichihara.kazuhiro@nihon-u.ac.jp}

\dedicatory{Dedicated to Professor Yo'av Rieck on his 60th birthday.}

\date{\today}



\begin{abstract}
We give an upper bound on the distance between a degeneracy slope for a very full essential lamination and a boundary slope of an essential surface embedded in a compact, orientable, irreducible, atoroidal 3-manifold with incompressible torus boundary. 
There are three applications: 
(i) We show that a degeneracy slope for a very full essential lamination in the exterior of a prime alternating knot is meridional. 
This gives an affirmative answer to part of a conjecture posed by Gabai and Kazez.
(ii) We obtain two bounds on boundary slopes for a hyperbolic knot in an integral homology sphere, at least one of which always holds: one concerning the denominators of boundary slopes, and the other concerning the differences between boundary slopes. 
This generalizes a result on Montesinos knots obtained by the author and Mizushima.
(iii) We obtain two bounds on exceptional surgery slopes for a hyperbolic knot in an integral homology sphere, at least one of which always holds: one concerning the denominators of such slopes, and the other concerning their range in terms of the genera of the knots. 
Both are actually conjectured by Gordon and Teragaito to always hold for hyperbolic knots in the 3-sphere.
\end{abstract}

\maketitle

\section{Introduction}

Let us start with fundamental terminology used in the paper. 
See \cite{Rolfsen} as a basic reference, for example.
As usual, by a \textit{slope}, we mean the isotopy class of a non-trivial unoriented simple closed curve on a torus.
The \textit{distance} between two slopes is defined as 
the minimal geometric intersection number of their representatives. 
(Note that this is not a distance function on the set of slopes.)
We use $\Delta(\gamma, \gamma')$ to denote the distance between the slopes $\gamma$ and $\gamma'$.
When a meridian-longitude system on a torus is fixed, 
slopes on the torus are parametrized by rational numbers (irreducible fractions) with $1/0$ in the usual way, and are identified accordingly.
In the following, when the torus is the boundary of the regular neighborhood of a knot in the 3-sphere, or in an integral homology sphere, we always adopt the standard meridian–preferred longitude system.
In this case, the distance between two slopes corresponding to $p/q$ and $r/s$ is calculated as $\Delta (p/q, r/s) = |ps - qr|$.

Let $F$ be a compact surface, which is always assumed to be connected in this paper, possibly non-orientable or with non-empty boundary, properly embedded in a 3-manifold $M$. 
Then, except when $F$ is a 2-sphere, it is called \emph{essential} if it is incompressible, 
$\partial$-incompressible, and not parallel to a surface on the boundary $\partial M$ of $M$.
%
On a torus boundary of a 3-manifold $M$, a slope is called a \textit{boundary slope} if it is represented by a boundary component of an essential surface properly embedded in $M$.
Throughout this paper, for convenience, we denote by $\chi(F)$ the Euler characteristic of 
a compact, possibly non-orientable, connected surface $F$, and by $|\partial F|$ the number of boundary components of $F$. 
When $F$ is orientable, we denote the genus of $F$ by $g(F)$.


A codimension one foliation $\lambda$ of a closed subset of a 3-manifold $M$ is called a \textit{lamination} in $M$.
The complement of a lamination $\lambda$ decomposes into connected components called \emph{complementary regions}. 
A lamination is called \emph{essential} 
if it contains no spherical or Reeb torus leaf and if its complement contains neither essential disks, monogons, nor spheres.
See \cite[Section 5]{GabaiOertel} 
and \cite[Chapter 6]{CalegariBook}. 

Based on \cite{Brittenham98} and \cite[Definition 8.5]{Gabai}, we extend the definition of a degeneracy slope of an essential lamination in a knot exterior to one in a compact orientable 3-manifold $M$ with torus boundary $T$ as follows. 
See also \cite[Subsection 2.1]{Ni25arxiv}. 
Suppose a meridian–longitude system on $\partial M$ is fixed. 
For an essential lamination $\lambda$ in $M$, 
(if possible) the \emph{degeneracy locus} $d(\lambda)$ of $\lambda$ is defined as $n(p,q)$ if there exist $n$ pairwise disjoint, properly embedded, non-parallel annuli in a closed complementary region such that one boundary component lies in a leaf of $\lambda$ and the other is a $p/q$ curve on $\partial M$. 
Furthermore, $n$ is set to be the maximal number of such annuli. 
The rational number $p/q$ is called a \emph{degeneracy slope} of $\lambda$. 
We remark that this definition is slightly different from that in \cite[Chapter 6]{CalegariBook}.


We recall here three known results on degeneracy slopes and boundary slopes.

One is \cite[Theorem 8.8]{Gabai}, stated as follows. 
Let $K$ be a knot in the 3-sphere $S^3$ and $\lambda$ an essential lamination in the exterior $E(K)$. 
If $d(\lambda) = n(1,0)$, then $1 \le n \le 4g(K) - 2$; otherwise, $d(\lambda) = 1(p,1)$ and $|p| \le 4g(K) - 2$, 
where $g(K)$ denotes the genus of $K$, that is, the genus of the minimal genus Seifert surface of $K$.
Note that if $d(\lambda) = 1(p,1)$, then $\Delta ( p/1 , 0/1) = p$ is the distance between the degeneracy slope and the longitudinal slope, 
that is, the boundary slope of a minimal genus Seifert surface of genus $g(K)$, which is known to be essential.
Thus, in the case $d(\lambda) = 1(p,1)$, the inequality can be rewritten as 
$\Delta(\delta, \beta) \le 4g(K) - 2$ for a degeneracy slope $\delta = p/1$ and the boundary slope $\beta$ of a minimal genus Seifert surface.


A closed essential surface in a 3-manifold is regarded as an essential lamination with a single compact leaf. 
In this case, the following result is shown in \cite{IchiharaOzawa02}. 
Let $K$ be a knot in $S^3$ with exterior $E(K)$, $F$ an essential surface properly embedded in $E(K)$, and $S$ a closed essential surface in $E(K)$ with an annulus $A$ connecting $S$ and $\partial E(K)$. 
Let $\beta$ denote the boundary slope of $F$ and $\delta$ the slope on $\partial E(K)$ determined by $A$. 
Then, 
$\Delta(\beta, \delta) \leq 2 (-\chi(F)) / |\partial F|$
holds. 
Note that if we regard $S$ as an essential lamination with a single compact leaf, the slope $\delta$ is a degeneracy slope of $S$. 
Also note that if the surface $F$ is orientable and has a single boundary component, then we have 
$2 (-\chi(F))/|\partial F| = 4 g(F) - 2$.

On the other hand, as with train tracks, laminations can be carried by branched surfaces, which allows one to study them combinatorially.
In the proof of \cite[Theorem 2.5]{Wu}, the following is essentially shown. 
Suppose $M$ is a hyperbolic manifold with torus boundary $T$. 
Let $B$ be an essential branched surface in $M$ with 
a regular neighborhood $N(B)$. (See the next section for details.)
Suppose that $M - \mathrm{Int}\,N(B)$ contains an essential annulus $A$ with one boundary on $\partial_h N(B)$ and the other a curve on $T$ of slope $\delta$.
Then a boundary slope $\beta$ of an essential surface of genus one (punctured torus) satisfies $\Delta(\delta, \beta) \le 2$.
Note that the slope $\delta$ corresponds to a degeneracy slope of an essential lamination if $B$ carries a (necessarily essential) lamination. 


Our following theorem is a partial extension of the three results above. 
It may be known to some experts, but as far as the author knows, there is no explicit statement in the literature.

\begin{Theorem}\label{Thm1}
Let $M$ be a compact, orientable, irreducible $3$-manifold with an incompressible torus boundary $\partial M$. 
For a boundary slope $\beta$ on $\partial M$ of an essential surface $F$ embedded in $M$, and a degeneracy slope $\delta$ on $\partial M$ of a very full essential lamination in $M$, the inequality $\Delta ( \beta , \delta ) \leq 2 (- \chi(F))/|\partial F|$ holds. 
In particular, if $F$ is orientable, then $\Delta ( \beta , \delta ) \leq 4 g(F) - 2$ holds.
\end{Theorem}


Here, an essential lamination in a $3$-manifold is called \textit{very full} if every complementary region is homeomorphic to one of the following two types:
(i) an ideal polygon bundle over the circle $S^1$, where the polygon has at least two ideal vertices; or 
(ii) a one-holed ideal polygon bundle over $S^1$, where the polygon has at least one ideal vertex. 
See \cite[Definition 6.42]{CalegariBook} for more details.
It is proved by Gabai and Mosher that any compact, orientable, irreducible, atoroidal $3$-manifold with incompressible torus boundary, in particular, a hyperbolic knot complement, contains a very full lamination.
See \cite[Theorem A]{GabaiMosher} and \cite[Theorem 6.47 (Gabai–Mosher)]{CalegariBook}.


We have three applications of the theorem as follows.

\subsection{}

A knot in $S^3$ is called \emph{alternating} if it admits a diagram in which over- and under-crossings alternate along the knot.  
We recall two results on degeneracy slopes of essential laminations in the exteriors of alternating knots.

It follows from \cite[Theorem 2 (Meridian Lemma)]{Menasco} and \cite[Theorem 1]{IchiharaOzawa00} that if $K$ is a prime alternating knot having a closed incompressible surface $S$ in its exterior $E(K)$, then there always exist embedded annuli in $E(K)$ connecting $S$ to $\partial E(K)$, and the slope on $\partial E(K)$ determined by the annuli is meridional.  
This means that, when we regard $S$ as an essential lamination with a single compact leaf, the degeneracy slope must be meridional.

For 2-bridge knots, which are well known to be alternating knots, the following was shown in \cite[Theorem 8]{GabaiKazez}.  
Let $K$ be a fibered 2-bridge knot in $S^3$.  
The stable lamination transverse to the fibration is known to be very full.  
Its degeneracy locus is then of the form $2(m - 1)\,(1, 0)$, where the positive integer $m$ is determined by the Murasugi sum decomposition of the fiber surface.  
In particular, the degeneracy slope is meridional.

It was raised as a conjecture in \cite[Conjecture 13]{GabaiKazez} that if $K$ is a non-torus fibered alternating knot, then $d(\lambda) = n(1, 0)$ holds for the stable lamination $\lambda$ transverse to the fibration and some integer $n > 1$.

The following gives a partial affirmative answer to the conjecture.

\begin{Theorem}[Degeneracy slopes for alternating knots]\label{Thm2}
Every degeneracy slope of a very full essential lamination in the exterior of a nontrivial prime alternating knot in the 3-sphere is meridional.
\end{Theorem}

The Gabai’s conjecture that degeneracy slopes of fibered alternating knots are meridional was recently proved in \cite[Corollary 6.18]{BoyerGordonHu24}, by using $\mathrm{Homeo}_+(S^1)$-representations of link groups. 
This theorem provides an alternative proof. 
See \cite{BoyerGordonHu24} for details.

Also, it implies the following corollary.

\begin{Corollary}\label{Cor-veering}
For every alternating fibered knot, the monodromy of the fibration is neither right-veering nor left-veering.     
\end{Corollary}

\begin{proof}
For a hyperbolic alternating fibered knot $K$, the stable lamination $\lambda$ transverse to the fibration is known to be very full.  
Let $d$ be the degeneracy slope for $\lambda$.  
Then, by \cite[Proposition 3.1]{HondaKazezMatic},  
the monodromy of the fibration is right-veering (resp. left-veering) if and only if $1/d > 0$ (resp. $1/d < 0$).  
Thus, the corollary follows from the theorem above.  
\end{proof}

This result was recently shown in \cite{BaldwinNiSivek} by using knot Floer homology. 
The corollary above provides an alternative proof. 
See \cite{BaldwinNiSivek} for details.











\subsection{}

Based on Hatcher-Thurston \cite{HatcherThurston} and 
Hatcher-Oertel \cite{HatcherOertel}, the following results were shown in \cite[Theorem 1.4]{IchiharaMizushima} about boundary slopes for Montesinos knots in $S^3$. 
(See \cite{IchiharaMizushima} for details.)
Let $K$ be a Montesinos knot of length at least three.
(i) 
Let $F$ be an essential surface of genus $g \ge 2$ properly embedded in the exterior of $K$ with boundary slope $p/q$. 
Then, $|q| \le (- \chi(F)) / |\partial F| + 1$ holds. 
In particular, if $F$ is orientable, then 
$q \le \max \{ g+1, 2g - 1 \}$ holds (\cite[Theorem 1.1]{IchiharaMizushima}).
(ii) 
Let $F_i$ be essential surfaces of genus $g_i \ge 2$ properly embedded in the exterior of $K$ with non-meridional boundary slopes $r_i$ for $i = 1, 2$, respectively. 
Then, $|r_1 - r_2| \le 4 ( g_1 + g_2 )$ holds 
(\cite[Theorem 1.4]{IchiharaMizushima}).

The next theorem shows that one of the results above always holds in general.

\begin{Theorem}[Bounds on boundary slopes]\label{Thm3}
Let $K$ be a hyperbolic knot in an integral homology sphere with exterior $E(K)$.  
Then one of the following must hold:
(i) For any essential orientable surface $F$ embedded in $E(K)$ with non-meridional boundary slope $p/q$, $|q| \le 2 g(F)$ holds.
(ii) For any pair of orientable essential surfaces $F$ and $F'$ embedded in $E(K)$ with non-meridional boundary slopes $r$ and $r'$, respectively, $|r - r'| \le 4 (g(F) + g(F') - 1)$ holds.
\end{Theorem}

See the following for related results: \cite[Theorem 2]{GordonLuecke1}, \cite[Theorem 2, 3]{Torisu}, \cite[Theorem 2]{MenascoThistlethwaite}, and \cite[Theorem 1.2]{MatignonSayari} about bounds on denominators of boundary slopes for knots, and also \cite[Theorem 5.2]{Rieck},  \cite[Theorem 1]{Torisu}, and \cite[Corollary 4.4]{ichihara2025euclideanlengthscullershalennorms} about bounds on the distances between boundary slopes for knots.

\subsection{}


Finally, we give an application to the study of exceptional surgeries on hyperbolic knots.  
See, for example, \cite{Boyer} for a survey. 

The following operation to create a 3-manifold from a given one and a given knot is called a \textit{Dehn surgery} on the knot:  
Remove an open tubular neighborhood of the knot,  
and glue a solid torus back.
Consider the slope on the peripheral torus of a knot $K$  which is represented by the curve identified with the meridian of the attached solid torus via the surgery.  
Then we can see that Dehn surgery on $K$ is characterized by this slope, which we call the \textit{surgery slope}.  
In the following, the 3-manifold obtained by Dehn surgery on a knot $K$ in $S^3$ along the slope $r$ is denoted by $K(r)$.


The well-known Hyperbolic Dehn Surgery Theorem due to Thurston \cite[Theorem 5.8.2]{Thurston1} states that each hyperbolic knot (i.e., a knot with hyperbolic complement) admits only finitely many Dehn surgeries yielding non-hyperbolic manifolds.  
In view of this, such finitely many exceptions are called \textit{exceptional surgeries}.

Regarding exceptional surgeries on knots in $S^3$, there are two conjectures, both are still open. 
One is about the denominators of exceptional surgery slopes: Every non-trivial exceptional surgery slope $p/q$ for a hyperbolic knot in $S^3$ satisfies $|q| \le 2$.  
See \cite[Problem 1.77 A(3)]{Kirbylist} for more explanation.
Another states: Every non-trivial exceptional surgery slope $r$ for a hyperbolic knot in $S^3$ satisfies $|r| \le 4 g(K)$.

The next theorem shows that one of the conjectures holds for each hyperbolic knot in an integral homology sphere. 

\begin{Theorem}[Bounds on exceptional surgery slopes]\label{Thm4}
Let $p/q$ be a non-trivial exceptional surgery slope for a hyperbolic knot in an integral homology sphere.  
Then at least one of the following always holds: (i) $|q| \le 2$, or (ii) $| p/q - r | \le 4 g(F)$ for every orientable essential surface $F$ with non-meridional boundary slope $r$, in particular, $|p/q| \le 4 g(K)$. 
\end{Theorem}

Note that it is not shown that at least one of the above statements always holds, i.e., it could be that for some knots the first statement holds, and for some knots the second, but no statement holds for all knots. 

Similar results were obtained in \cite[Theorem 1.1]{Ni25arxiv} and \cite[Corollary 1.12]{BaldwinNiSivek} for hyperbolic fibered knots in $S^3$.

By virtue of the Geometrization of 3-manifolds  
established in \cite{Perelman1, Perelman2, Perelman3}, exceptional surgeries are classified into those yielding reducible, toroidal, or (small) Seifert fibered manifolds.  
Among these, surgeries yielding spherical manifolds, that is, 3-manifolds with finite fundamental groups, in particular lens spaces, have been extensively studied.  
For these, we have the following.  

\begin{Corollary}\label{Cor_CycGenus}
Let $K$ be a hyperbolic knot in $S^3$ and  
$b$ a non-meridional boundary slope of an orientable essential surface $F$ embedded in the exterior $E(K)$ of $K$. 
Suppose that the degeneracy locus of a very full essential lamination in $E(K)$ is not $1(1,0)$. 
Then, if the 3-manifold $K(r)$ obtained by $r$-surgery on $K$ for a non-meridional slope $r$ is a lens space or $\pi_1 ( K(r) )$ is finite, then $| r - b | \leq 4 g(F) - 1$, in particular, $| r | \leq 4 g(K) - 1$.  
\end{Corollary}

\begin{proof}
Let $\gamma = r $ be a slope such that $K(r)$ is a lens space or $\pi_1 ( K(r) )$ is finite. 
We consider the degeneracy slope $\delta$ of a very full essential lamination $\lambda$ in the exterior $E(K)$ of $K$. 
As stated before, by \cite[Theorem A]{GabaiMosher}, 
there always exists a very full essential lamination in $E(K)$, and so the degeneracy slope $\delta$ must exist. 
See also \cite[{Theorem 6.47. (Gabai–Mosher)}]{CalegariBook}. 

When $\delta=1/0$, equivalently, the corresponding degeneracy locus is $n(1,0)$, if $n \ge 2$, the lamination $\lambda$ remains essential in the manifold obtained by surgery along any $\gamma \ne \delta =1/0$ by \cite{GabaiOertel}. 
Moreover, for any $\delta$, if $\Delta (\delta, \gamma) \ge 2$, $\lambda$ remains essential in the manifold obtained by surgery along $\gamma$ by \cite[Theorem 6.48. (Gabai–Mosher)]{CalegariBook}. 

Since the universal cover of a compact orientable 3-manifold containing an essential lamination is homeomorphic to $\mathbb{R}^3$ (\cite[Theorem 6.1]{GabaiOertel}), in particular, the fundamental group of the manifold must be infinite, any lens space or a 3-manifold with finite fundamental group cannot contain an essential lamination. 
This implies that, under the assumption that the degeneracy locus is not $1(1,0)$, $\delta \ne 1/0$ and $\Delta( \delta, \gamma ) \le 1$. 

Let $\beta = b = b_1/b_2$ be a non-meridional boundary slope of an orientable essential surface $F$ embedded in the exterior of $K$, and set $r = r_1/r_2$, $\delta=d_1/d_2$ with $d_2 \ne 0$. 
Then, we have the following from Theorem~\ref{Thm1}. 
\begin{align*}
| r - b | & \le 
| r_1/r_2 - d_1/d_2 | + | d_1/d_2 - b_1/b_2 | 
 =
\frac{\Delta( r_1/r_2 , d_1/d_2)}{r_2 d_2} 
+ \frac{\Delta ( d_1/d_2 , b_1/b_2 )}{d_2 b_2} \\
& \le
\Delta( \gamma , \delta ) 
+ \Delta ( \delta , \beta ) 
 \le 1 + 4g(F)-2 \leq 4 g(F) - 1.
\end{align*}

If we consider a minimal genus Seifert surface, which is known to be essential, as $F$, we obtain $| r | \leq 4 g(K) -1$. 
\end{proof}

See the following for related results; \cite[Theorem 1]{GordonLuecke1}, \cite[Theorem 1.1]{BoyerZhang}, \cite{CGLS}, and \cite{GordonLuecke2,GordonLuecke3}  about the bounds on the denominators of exceptional surgery slopes, and also, \cite[Theorem 1.1]{GodaTeragaito}, \cite[Corollary 8.5]{KronheimerMrowkaOzsvathSzabo}, \cite[Theorem 1]{Rasmussen} (for les space surgery), 
\cite[Proposition 1.3 (i)]{MatignionSayari03} (for reducible surgery), 
\cite[Theorem 1.2, Theorem 1.3]{Teragaito}, \cite[Theorem 1]{EudaveMunozGuzmanTristan}, 
\cite{IchiharaTeragaito1, IchiharaTeragaito2} (for toroidal surgery), and also 
\cite{Ichihara01, IchiharaRIMS}, \cite[Theorem 5.1(B)]{IchiharaOzawa02}, 
about the bounds on the exceptional surgery slopes in terms of the genera of knots.

\subsection*{Acknowledgments}
The auhtor would like to thank Tsuyoshi Kobayashi and Yoa'v Rieck for helpful conversations. 
He also thanks to Steven Boyer for letting him know the reference \cite{BoyerGordonHu24} regarding Theorem~\ref{Thm2} and the anonymous referee for a careful reading and helpful
comments.
This work was partially supported by JSPS KAKENHI Grant Number JP22K03301.

\section{Distances between degeneracy slopes and boundary slopes}

We first prepare basic terminology on essential laminations and essential branched surfaces. 
See \cite{Calegari1, CalegariBook} for details, and also \cite[Definitions 1.1]{GabaiOertel}, \cite{Gabai}, and \cite{Wu}. 

A \emph{branched surface} $B$ in a 3-manifold $M$ is defined as a 2-complex with a $C^1$ combing along the 1-skeleton, giving it a well-defined tangent space at every point, and generic singularities (\cite[Definition 6.16]{CalegariBook}). 

By $N(B)$, we denote the normal bundle over $B$ in $M$ foliated by intervals transverse to $B$. 
For a branched surface $B$ in a 3-manifold $M$, its exterior $E(B)$ is defined as $E(B) = M - \mathrm{Int}\, N(B)$. 

If $B$ is a branched surface in $\mathrm{Int}\, M$, then $E(B)$ has a natural cusped manifold structure. 
That is, the \emph{horizontal boundary} of $E(B)$ is defined as $\partial_h E(B) = \partial_h  N(B) \cup \partial M$, and the \emph{vertical boundary} of $E(B)$ as $\partial_v E(B) =
\partial_v N(B)$, where $\partial_h N(B)$ and $\partial_v N(B)$ are the horizontal and vertical boundaries of $N(B)$.

It is proved by \cite[Proposition 4.5]{GabaiOertel} that every essential lamination is fully carried by a branched surface with the following properties, which is called an \emph{essential branched surface}. 
\begin{enumerate}
    \item $B$ has no disk of contact, i.e., no disk $D \subset N(B)$ such that $D$ is transverse to the $I$–fibers of $N(B)$ and $\partial D \subset \partial_v N(B)$.
    \item There is no Reeb component, i.e., $B$ does not carry a torus that bounds a solid torus in $M$.
    \item No component of $\partial_h N(B)$ is a sphere.
    \item $B$ fully carries an essential lamination. 
    \item $E(B) = M - \mathrm{Int}(N(B))$ contains no essential surface $F$ with $\chi_c(F) > 0$. This is equivalent to the following:
    \begin{enumerate}
        \item $M - B$ is irreducible,
        \item $\partial_h N(B)$ is incompressible in $E(B)$,
        \item $E(B)$ has no monogon, i.e., no disk $D \subset M - \mathrm{Int}(N(B))$ with $\partial D = D \cap N(B) = \alpha \cup  \beta$, where $\alpha \subset  \partial_v N(B)$ is in an interval fiber of $\partial_v N(B)$ and $\beta \subset \partial_h N(B)$.
    \end{enumerate}
\end{enumerate}
Remark that the four conditions (1), (2), (3), (4) above are intrinsic, while the condition (5) is extrinsic.

Here, $\chi_c(F)$ denotes the \emph{cusped Euler characteristic} of a \emph{cusped surface} $F$ appearing as a complementary region of a train track on a surface or as a properly embedded surface in $E(B)$ such that $F \cap \partial_v N(B)$ consists of a set of essential arcs in $\partial_v N(B)$, each of which we call a \emph{cusp} on the boundary $\partial F$. 
Precisely, $\chi_c(F)$ is defined as $\chi_c(F) = \chi(F) - C(F)/2$, where $\chi(F)$ is the Euler characteristic of $F$, and $C(F)$ is the number of cusps on $\partial F$. 
This quantity has an additivity property:  
If $\tau$ is a train track on $F$, then $\chi(F) = \sum \chi_c(F_i)$, where $F_i$ are the components of $F - \mathrm{Int} N(\tau)$. 

Here, we recall the following fact. 

\begin{Lemma}[{\cite[Theorem 2.2]{Wu}}]\label{Lem_Wu}
Suppose $B$ is an essential branched surface which fully carries a lamination $\lambda$, and suppose $F$ is an essential surface in $M$. 
Then there is an essential branched surface $B'$ which is a splitting 
of $B$, and a surface $F'$ isotopic to $F$, such that $F' \cap B'$ is an essential train track on $F'$.    
\end{Lemma}

Here, a train track, which is a branched compact 1-manifold on a surface, is called \emph{essential} if it has no $0$-gons (i.e., disks) or monogons in its complement. 

We remark that the surfaces considered in \cite{Wu} are assumed to be orientable, but the lemma holds for essential non-orientable surfaces as well.  
See the proof of \cite[Theorem 2.2]{Wu}.

The above lemma is very useful in our situation, but it is not sufficient for our purpose, since it only considers the interior of the essential surface $F$; we need to consider the intersection $B \cap F$ in a neighborhood of the boundary of $F$. 

\begin{proof}[Proof of Theorem~\ref{Thm1}]
Let $M$ be a compact, orientable, irreducible $3$-manifold with incompressible torus boundary $\partial M$. 

Suppose that there exists an essential surface $F$ properly embedded in $M$ with boundary $\partial F$ having boundary slope $\beta$ on $\partial M$, and that there exists a very full essential lamination $\lambda$ in $M$ with degeneracy slope $\delta$ on $\partial M$. 
Let $\Delta$ denote the distance $\Delta(\beta, \delta)$. 

By \cite[Proposition 4.5]{GabaiOertel}, there exists an essential branched surface $B$ that fully carries the lamination $\lambda$. 
After a small perturbation, we may assume that $B \cap F$ is a train track, say $\tau$. 

\begin{Claim}
An essential branched surface $B'$ and an essential surface $F'$ are obtained from $B$ and $F$ by splittings and isotopies such that $F' \cap B'$ is an essential train track $\tau'$ on $F'$, and that, for the complementary region $X'$ of $N(B')$ containing $\partial M$, every component of $F' \cap X'$ that contains a component of $\partial F'$ is an essential annulus in $X'$. 
\end{Claim}

\begin{proof}
First, by Lemma~\ref{Lem_Wu}, we obtain an essential branched surface, say $B_1$, from $B$ by splitting, and obtain an essential surface, say $F_1$, by isotopies such that $F_1 \cap B_1$ is an essential train track $\tau_1$ on $F_1$, that is, $\tau_1$ has no complementary disks or monogons. 
We remark that in the procedure to obtain $\tau'$ from $\tau$, the number of switches must decrease when $\tau$ is inessential. 
See \cite[Proof of Theorem 2.2 and Figure 2.1]{Wu}. 

Let $X_1$ be the complementary region of $N(B_1)$ containing $\partial M$, and consider a component, say $S$, of $F_1 \cap X_1$ containing $\partial F_1$. 
By further isotopy of $F_1$ if necessary, $S$ is assumed to be incompressible and not boundary-parallel in $X_1$, for $F_1$ is essential in $M$ and $M$ is irreducible. 
Note that, since $\lambda$ is very full, $X$ is homeomorphic to $T^2 \times I$, and so is $X_1$. 
Thus, $S$ is essential in $X_1$ if and only if $S$ is a vertical annulus in $X_1 \simeq T^2 \times I$. 

Suppose otherwise; $S$ is inessential in $X_1$. 
Then, $S$ must be $\partial$-compressible in $X_1$. 
Since $S$ intersects with $\partial M$, and is not a vertical annulus, there exists a $\partial$-compressing disk for $S$ whose boundary consists of an arc $\alpha$ on $S$ and an arc on $\partial M$. 
Note that, since $F_1$ is $\partial$-incompressible, the arc $\alpha$ cobounds a disk on $F_1$ together with an arc on $\partial F_1$, which may contain some connected components of $\tau_1$ on $F_1$. 

Perform the $\partial$-compression of $S$ in $X_1$, which is achieved by isotopy of $F_1$, since $F_1$ is essential, $M$ is irreducible, and $\partial M$ is incompressible.
Then, $\tau_1$ is transformed into a new one with a reduced number of connected components, since the arc $\alpha$ on $S$ separates connected components of $\tau_1$ on $F_1$. 
After this transformation, if $\tau_1$ becomes inessential, then, by using Lemma~\ref{Lem_Wu} again, split and isotope $B_1$ and $F_1$ so that the intersection gives an essential train track. 

Repeat this procedure as much as possible. 
We remark that the procedure must terminate, since the number of switches of the train tracks decreases in each step. 
Thus, we finally get the desired $B'$ and $F'$. 
That is, $F' \cap B'$ is an essential train track $\tau'$ on $F'$, and, for the complementary region $X'$ of $N(B')$ containing $\partial M$, every component of $F' \cap X'$ containing a component of  $\partial F'$ is an essential (vertical) annulus in $X'$. 
\end{proof}

\begin{Claim}
A surface $F''$, having the same property as $F$ and the properties achieved in Claim~1, is obtained from $F'$ by isotopy so that, as a cusped surface, every component of $F'' \cap X'$ containing a component of $\partial F''$ is a one-holed ideal polygon with at least $\Delta$ cusps. 
\end{Claim}

\begin{proof}
Since $\lambda$ is a very full essential lamination, there exists an annulus, say $A$, running from $\partial_h N(B')$ to $\partial M$. 
This annulus $A$ must be essential in $X'$, since it is vertical in $X' \simeq T^2 \times I$ and $\partial_h N(B')$ is incompressible in $X'$. 

We isotope $F'$, fixing $F' \cap \partial_v N(B')$, to obtain an essential surface $F''$ so that the number of components of the intersection $(F'' \cap X') \cap A$ becomes minimal, while keeping the properties achieved in Claim~1. 
Consider the intersection of $A$ with a component $S$ of $F'' \cap X'$ containing a component of $\partial F''$, which is an essential annulus in $X'$ by Claim~1. 

Then, $S \cap A$ consists of arcs, say $\alpha_i$'s, which are essential in both $S$ and $A$, and of inessential arcs, say $\alpha'_j$'s, on $S$ parallel to $\partial N(B')$. 
That is, each $\alpha'_j$ cobounds a bigon on $S$ together with an arc on $\partial N(B')$. 
Moreover, each of such bigons on $S$ contains at least one cusp, i.e., an arc of $S \cap \partial_v N(B')$, since bigons without cusps were removed by the isotopy of $F'$ in $X'$ by the essentiality of $S$, $A$, and $B'$. 
Note that the number of essential arcs $\alpha_i$'s is at least $\Delta$, since the intersection of each component of $\partial F''$ and the curve $A \cap \partial M$ is at least $\Delta$.  

We here claim that each rectangular region $R$ on $S$ bounded by two adjacent $\alpha_i$'s contains at least one cusp, i.e., an arc of $S \cap \partial_v N(B')$. 
If $R$ contains $\alpha'_j$'s, then, as observed above, it contains at least one cusp. 
Suppose otherwise that $R$ contains no $\alpha'_j$'s, and suppose $R$ contains no cusps to induce a contradiction. 
That is, $R \cap \partial N(B') \subset \partial_h N(B')$ holds. 
Note that, by the minimality of $S \cap A$, there are no bigons on $\partial M$ bounded by arcs on $\partial S$ and $\partial A$.  
In this situation, by using $R$, we find another essential annulus $A'$ in $X'$ whose one boundary component lies on $\partial M$ and the other on $\partial_h N(B')$ as follows. 
Consider the arcs $R \cap A$ on $A$, which is a pair of essential arcs $\alpha_i$'s on $A$. They cobounds a rectangular region $R'$ on $A$. (Actually there are two such regions, but either one can be used.) 
The union $R \cup R'$ gives a vertical, thus essential, annulus $A'$ in $X'$ whose one boundary component lies on $\partial M$ and the other on $\partial_h N(B')$, since $R \cap \partial N(B') \subset \partial_h N(B')$. 
By construction and the minimality of $S \cap A$, this $A'$ has a different boundary slope on $\partial M$ from that of $A$. 
However, this contradicts the uniqueness of the degeneracy slope for the very full essential lamination $\lambda$, which follows from the fact that $X'$ is a one-holed ideal polygon bundle over $S^1$. 
See also \cite[Remark 8.6]{Gabai}. 

Since the number of arcs $\alpha_i$'s is at least $\Delta$, it follows that the surface $F''$ obtained from $F'$ by the isotopies, which has the same property as $F$ and has the properties achieved in Claim~1, satisfies that, as a cusped surface, every component of $F'' \cap X'$ containing $\partial F''$ is a one-holed ideal polygon with at least $\Delta$ cusps.  
\end{proof}

Now we consider the cusped Euler characteristic $\chi_c(F'')$ of $F''$, which is equal to $\chi_c(F) = \chi(F)$.  
By additivity, for $X'$ in the proof of Claim~2, $\chi_c(F'')$ is equal to the sum of $\chi_c(F'' \cap X')$ and the cusped Euler characteristic of the remaining subsurface.  
By essentiality of $\tau'$ on $F''$, the latter must be non-positive.  
Thus, $\chi_c(F'') \leq \chi_c(F'' \cap X')$ holds.  

Since every component of $F'' \cap X'$ containing a component of $\partial F''$ is a one-holed ideal polygon with at least $\Delta$ cusps, the cusped Euler characteristic of each component is at most $-\Delta/2$.  
This implies that $\chi_c(F'' \cap X') \leq (-\Delta/2) \cdot |\partial F|$, and hence,  
\[
\chi(F) = \chi_c(F'') \leq \chi_c(F'' \cap X') \leq (-\Delta/2) \cdot |\partial F|.
\]

As a consequence, we obtain  
\[
\Delta(\beta, \delta) \leq 2 \frac{-\chi(F)}{ |\partial F|}.
\]

If $F$ is orientable, since $ \chi(F) = 2 - 2 g(F) - |\partial F| $ and $|\partial F| \ge 1$, we have;
\[
\Delta ( \beta , \delta ) 
\le 2 \frac{- \chi(F)}{|\partial F|}
 = \frac{ 4 g(F) - 4}{|\partial F|} + 2 
\le 4 g(F) -2. 
\]
\end{proof}

\section{Applications}\label{sec;applications}

In this section, we provide proofs of the three theorems, Theorem~\ref{Thm2}, \ref{Thm3}, and \ref{Thm4} stated in Section~1.

\subsection{Degeneracy slopes for alternating knots}

\begin{proof}[Proof of Theorem~\ref{Thm2}]
Let $K$ be a non-trivial prime alternating knot in $S^3$ with the exterior $E(K)$ in $S^3$. 
Then, by \cite{Menasco}, $E(K)$ is a compact, orientable, irreducible, atoroidal 3-manifold with incompressible torus boundary. 
Thus, for a degeneracy slope $\delta$ of a very full essential lamination in $E(K)$, Theorem~\ref{Thm1} is applicable to $E(K)$. 

Suppose, for a contradiction, that $\delta$ is non-meridional. 
Then, since the meridional surgery on $K$ yields $S^3$, $\delta$ is integral, i.e., $\delta = p/1$ for some integer $p$, by \cite[{Theorem 1.6(1)}]{Wu} plus the first few lines from the proof of \cite[Theorem 2.5]{Wu}. 

Consider the checkerboard surfaces $F_1$, $F_2$ for the reduced alternating diagram of $K$.
These are known to be essential by \cite{Aumann, DelmanRoberts}. 
Let $r_1$ and $r_2$ denote the boundary slopes of $F_1$ and $F_2$, which are both integral slopes. 
By simple calculations, 
we have 
\[
\Delta(r_1, r_2) = |r_1 - r_2| = 2\,\mathrm{cr}(K)
, \quad 
\chi(F_1) + \chi(F_2) = 2 - \mathrm{cr}(K).
\] 
Here, $\mathrm{cr}(K)$ denotes the minimal crossing number of $K$, which is actually equal to the number of crossings in the reduced alternating diagram. 
For example, see \cite{IchiharaMizushima2}.
Hence, the following holds;
\[
| r_1 - r_2 | = 2\,( (-\chi (F_1) )+ ( -\chi (F_2) ) ) + 4 .
\]

On the other hand, by Theorem~\ref{Thm1}, 
together with each $F_i$ having connected boundary, 
we must have 
\[
| r_i - p | = \Delta ( r_i/1 , p/1 ) \leq 2 \, (-\chi (F_i) ).
\]
This implies that 
\[
| r_1 - r_2 | \le 
| r_1 - p | + | p - r_2 | 
\leq 2 \big( (-\chi (F_1) ) + (-\chi (F_2) ) \big),
\]
contradicting the above.
\end{proof}

\subsection{Bounds on boundary slopes}


\begin{proof}[Proof of Theorem~\ref{Thm3}]
Let $K$ be a hyperbolic knot with the exterior $E(K)$ in an integral homology sphere $\Sigma$. 
We fix a meridian-preferred longitude system on $\partial E(K)$. 

Since $K$ is hyperbolic, the exterior $E(K)$ is 
a compact, orientable, irreducible, atoroidal 3-manifold with incompressible torus boundary. 
Then, as stated in the Introduction, by \cite[Theorem A]{GabaiMosher} 
and \cite[{Theorem 6.47. (Gabai–Mosher)}]{CalegariBook}, 
the exterior $E(K)$ contains a very full essential lamination. 
We consider its degeneracy slope $\delta$ on the boundary $\partial E(K)$. 

Suppose that $\delta$ is meridional, that is, $\delta=1/0$. 
Let $\beta = p/q$ with $q \neq 0$ be the non-meridional boundary slope on $\partial E(K)$ of an essential surface $F$ embedded in $E(K)$. 
Then, by Theorem~\ref{Thm1}, we have
\[
\Delta (\delta,\beta) = \Delta (1/0,p/q) = |q| \le 2 \frac{- \chi(F)}{|\partial F|}. 
\]
Note that if $F$ is orientable, $|\partial F| \ge 2$ holds. 
Because the image of the boundary map $\partial : H_2 ( E(K) , \partial E(K) ) \to H_1 ( \partial E(K))$ is a Lagrangian, that is, any two element in the image intersect algebraically zero times, hence when $F$ is orientable, we should have $\partial (F) = 0 \in H_1 ( \partial E(K))$ and so $| \partial F|$ is even.

Thus, we obtain 
\[
|q| \le 
2 \frac{- \chi(F)}{|\partial F|} 
= \frac{4 g(F) - 4}{|\partial F|} + 2 \le 2 g(F). 
\]

Suppose next that $\delta = d_1/d_2$ is non-meridional, i.e., $d_2 \neq 0$. 
Let $r = r_1/r_2$ and $r' = r'_1/r'_2$ be non-meridional boundary slopes of orientable essential surfaces $F$ and $F'$ embedded in $E(K)$, respectively. 
Then, by Theorem~\ref{Thm1}, the following holds:
\begin{align*}
|r_1/r_2 - r'_1/r'_2| &\le 
|r_1/r_2 - d_1/d_2| + |d_1/d_2 - r'_1/r'_2| \\
&= \frac{\Delta(r_1/r_2, d_1/d_2)}{r_2 d_2} + \frac{\Delta(d_1/d_2, r'_1/r'_2)}{d_2 r'_2} \\
&\le \Delta(r_1/r_2, d_1/d_2) + \Delta(d_1/d_2, r'_1/r'_2) \\
&\le (4g(F) - 2) + (4g(F') - 2) = 4(g(F) + g(F') - 1).
\end{align*}

\end{proof}

\subsection{Bounds on exceptional surgery slopes}


\begin{proof}[Proof of Theorem~\ref{Thm4}]
Let $K$ be a hyperbolic knot with exterior $E(K)$ in an integral homology sphere. We fix a meridian-preferred longitude system on $\partial E(K)$.

As in the proof of Theorem~\ref{Thm3}, since $K$ is hyperbolic, the exterior $E(K)$ is a compact, orientable, irreducible, atoroidal 3-manifold with incompressible torus boundary.
Thus, by \cite[Theorem A]{GabaiMosher} 
and \cite[Theorem 6.47 (Gabai–Mosher)]{CalegariBook}, the exterior $E(K)$ contains a very full essential lamination. We consider its degeneracy slope $\delta$ on the boundary $\partial E(K)$.

Let $\gamma = p/q$ be a non-trivial exceptional surgery slope for $K$. 
Then, by \cite[Theorem 2.5]{Wu} together with the Geometrization of 3-manifolds, we have $\Delta(\gamma, \delta) \le 2$.

\noindent
(i) In the case that $\delta$ is meridional, that is, $\delta=1/0$, then by Theorem~\ref{Thm1}, we have
\[
\Delta (\gamma, \delta) = \Delta (p/q, 1/0) = |q| \le 2 .
\]

\noindent
(ii) Consider the case that $\delta = d_1/d_2$ is non-meridional, i.e., $d_2 \ne 0$. 

For an orientable essential surface $F$ with non-meridional boundary slope $r = r_1/r_2$, applying Theorem~\ref{Thm1}, we obtain 
\begin{align*}
| p/q - r_1/r_2 | & \le 
| p/q - d_1/d_2 | + | d_1/d_2 - r_1/r_2 | \\
& =
\frac{\Delta( p/q , d_1/d_2 )}{q d_2} 
+ \frac{\Delta ( d_1/d_2 , r_1/r_2 )}{d_2 r_2} \\
& \le
\Delta( p/q , d_1/d_2 ) + \Delta ( d_1/d_2 , r_1/r_2 ) \\
& \le 2 + (4g(F)-2) \leq 4 g(F) .
\end{align*}

In particular, by considering the longitudinal slope, that is, $0/1$, of the minimal genus Seifert surface of genus $g(K)$, which is known to be essential, we have 
\[
|p/q| \le 4 g(K) .
\]
\end{proof}


\bibliographystyle{amsalpha}

\bibliography{ref}

\end{document}